\documentclass[12 pt]{amsart}
\usepackage{amscd,amssymb,amsthm}
\usepackage[arrow,matrix]{xy}
\usepackage{graphicx}
\usepackage{cite}
\oddsidemargin=0.3in \evensidemargin=0.3in \theoremstyle{plain}
\newtheorem{thm}[subsection]{Theorem}
\newtheorem{lem}[subsection]{Lemma}

\newtheorem{cor}[subsection]{Corollary}

\theoremstyle{definition}
\newtheorem{rk}[subsection]{Remark}
\newtheorem{definition}[subsection]{Definition}

\numberwithin{equation}{section} 

\begin{document}
\title [\texttt{Bond Incident Degree (BID) Indices of Polyomino Chains: A Unified Approach}]{Bond Incident Degree (BID) Indices of Polyomino Chains: A Unified Approach}
\author{Akbar Ali$^{\dag,\ddag}$, Zahid Raza$^{\dag}$ and Akhlaq Ahmad Bhatti$^{\dag}$}
 \address{$^{\dag}$Department of Mathematics \\ National University of Computer and Emerging Sciences, B-Block, Faisal Town, Lahore, Pakistan.}
 \address{$^{\ddag}$Department of Mathematics \\ University Of Gujrat, Hafiz Hayat Campus, Gujrat, Pakistan.}
\email{akbarali.maths@gmail.com,zahid.raza@nu.edu.pk,akhlaq.ahmad@nu.edu.pk}



\thanks{}

\begin{abstract}

This work is devoted to establish a general expression for calculating the bond incident degree (BID) indices of polyomino chains and to characterize the extremal polyomino chains with respect to several well known BID indices. From the derived results, all the results of [M. An, L. Xiong, Extremal polyomino chains with respect to general Randi\'{c} index, \textit{J. Comb. Optim.} (2014) DOI 10.1007/s10878-014-9781-6], [H. Deng, S. Balachandran, S. K. Ayyaswamy, Y. B. Venkatakrishnan, The harmonic indices of polyomino chains, \textit{Natl. Acad. Sci. Lett.} \textbf{37}(5), (2014) 451-455], [Z. Yarahmadi, A. R. Ashrafi and S. Moradi, Extremal polyomino chains with respect to Zagreb indices, \textit{Appl. Math. Lett.} \textbf{25} (2012) 166-171], and also some results of [J. Rada, The linear chain as an extremal value of VDB topological indices of polyomino chains, \textit{Appl. Math. Sci.} \textbf{8}, (2014) 5133-5143], [A. Ali, A. A. Bhatti, Z. Raza, Some vertex-degree-based topological indices of polyomino chains, \textit{J. Comput. Theor. Nanosci.} \textbf{12}(9), (2015) 2101-2107] are obtained as corollaries.

\end{abstract}
\maketitle
\section{Introduction}

According to the International Union of Pure and Applied Chemistry (IUPAC) Recommendations 1997 \cite{w1}, a \textit{topological index} is a numerical value associated with chemical constitution for correlation of chemical structure with various physical properties, chemical reactivity or biological activity. Nowadays, there are many topological indices that have found applications in chemistry \cite{d2}, in computational linguistics \cite{l1} and in computational biology \cite{d3}. A large number of such indices depend only on vertex degree of the molecular graph \cite{f1,c1} and are known as degree-based topological indices. In the present study, we are concerned with \textit{bond incident degree (BID) indices} (a subclass of degree-based topological indices) whose general form \cite{d4,f5,g1,r1,v2,VuDu11} is:
\begin{equation}\label{z}
TI=TI(G)=\displaystyle\sum_{uv\in E(G)}f(d_{u},d_{v})=\displaystyle\sum_{1\leq a\leq b\leq \Delta(G)}x_{a,b}(G).\theta_{a,b} \ ,
\end{equation}
where $uv$ is the edge connecting the vertices $u$ and $v$ of the graph $G$, $d_{u}$ is the degree of the vertex $u$, $E(G)$ is the edge set of $G$, $\Delta(G)$ is the maximum degree in $G$, $\theta_{a,b}$ is a non-negative real valued function depending on $a,b$ and $x_{a,b}(G)$ is the number of edges in $G$ connecting the vertices of degrees $a$ and $b$.
The (general) Randi$\acute{c}$ index \cite{b1,r3}, harmonic index \cite{f2}, (general) atom-bond connectivity index \cite{e3,x1}, (general) sum-connectivity index \cite{z2,z3}, first geometric-arithmetic index \cite{v1}, augmented Zagreb index \cite{f4}, Albertson \cite{a1} index and Zagreb indices \cite{e1,g22,g3,g4} are the special cases of (\ref{z}). Besides these, there are many other indices of the form (\ref{z}). Details about the BID indices can be found in the recent review \cite{g1}, papers \cite{Vu10,VuDu11,f5,g2} and references cited therein.

A polyomino system is a finite 2-connected plane graph such that each interior face is surrounded by a regular square of side length one. In a polyomino system, two squares are adjacent if they share an edge. For the history and details about polyomino system see for example \cite{g11,k1}. A polyomino system in which every square is adjacent with at most two squares is called a polyomino chain. The problem of characterizing the extremal polyomino chains with respect to BID indices over the set of all polyomino chains with fixed number of squares has attracted substantial attention from researchers in recent years. For instance, Yarahmadi \textit{et al.} \cite{z} determined extremal polyomino chains with respect to first and second Zagreb indices. Deng \textit{et al.} \cite{deng} characterized the extremal polyomino chains with respect to harmonic index. An and Xiong \cite{An2}, recently determined extremal polyomino chains for the general Randi\'{c} index. Rada \cite{r4} recently proved that the linear chain has the extremal value for many well known BID indices. In \cite{AA2,y2}, the same problem was addressed for some other BID indices. In this paper, an efficient closed form formula to calculate the BID indices of polyomino chains is given and by making use of this formula the extremal polyomino chains with respect to several BID indices are characterized, and thereby all the results reported in \cite{An2,deng,z} and also some results of \cite{y2,r4} are generalized.

\section{Main Results}
To establish the main results, we need some concepts for a polyomino chain. A square adjacent with only one (respectively two) other square(s) is called terminal (respectively non-terminal). By a kink, we mean a non-terminal square having a vertex of degree 2. A polyomino chain without kinks is called \textit{linear chain} (see the Figure \ref{fig:2}). A polyomino chain consisting of only kinks and terminal squares is known as zigzag chain (see the Figure \ref{fig:3}). A segment is a maximal linear chain in a polyomino chain, including the kinks and/or terminal squares at its ends. The number of squares in a segment $S$ is called its length and is denoted by $l(S)$. Two segments are adjacent if they share a square. For any segment $S$ of a polyomino chain with
$n\geq3$ squares, $2\leq l(S)\leq n$. It can be easily seen that\\ \\
\textit{(1).} A polyomino chain is linear if and only if it has only one segment. \\ \\
\textit{(2).} A polyomino chain with $n\geq3$ squares is zigzag if and only if it has $n-1$ segments.\\

\begin{figure}
  \centering
    \includegraphics[width=2.45in, height=0.55in]{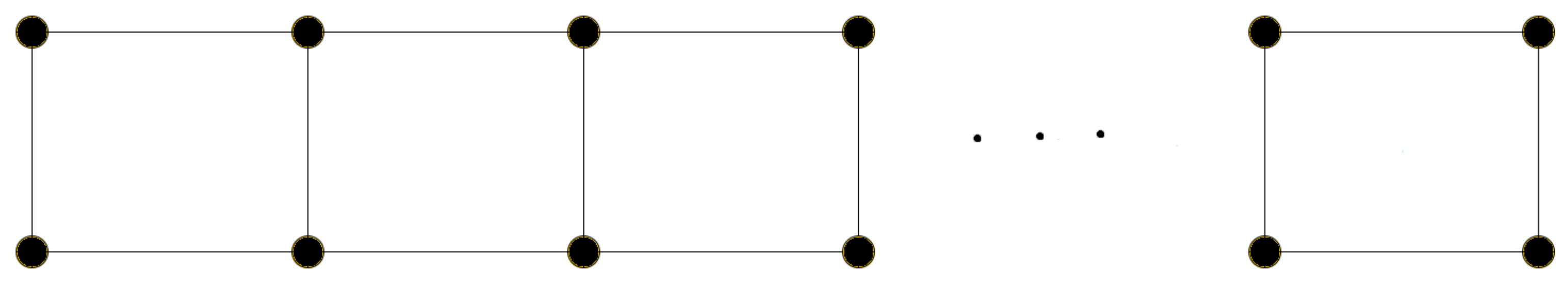}
    \caption{A linear polyomino chain}
    \label{fig:2}
      \end{figure}

\begin{figure}
    \centering
    \includegraphics[width=2.95in, height=2in]{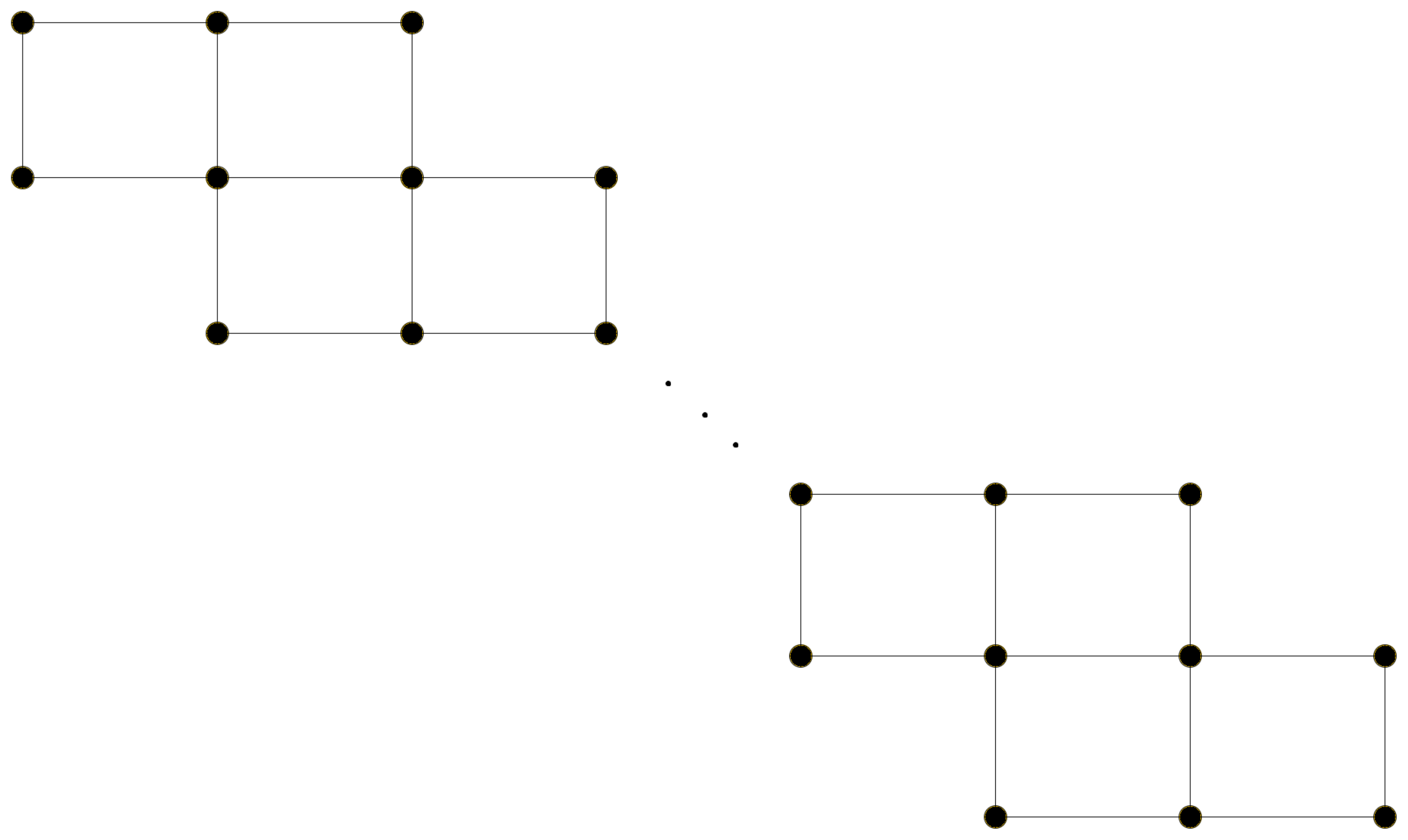}
    \caption{A zigzag polyomino chain}
    \label{fig:3}
     \end{figure}

A polyomino chain $B_{n}$ consists of a sequence of segments $S_{1}, S_{2},S_{3},...,S_{s}$ with lengths $l(S_{i})=l_{i}$ $(1\leq i\leq s)$ such that $\displaystyle\sum_{i=1}^{s}l_{i} = n + s - 1$. The vector $l=(l_{1},l_{2},...,l_{s})$ is called length vector. We also need the following definitions:
\begin{definition}\cite{z}
For $2\leq i\leq s-1$ and $1\leq j\leq s$,
$$\alpha_{i}=\alpha(S_{i})=
\begin{cases}
1 & \text{if } l(S_{i})=2\\
0 & \text{if } l(S_{i})\geq3
\end{cases}$$
$$\beta_{j}=\beta(S_{j})=
\begin{cases}
1 & \text{if } l(S_{j})=2\\
0 & \text{if } l(S_{j})\geq3
\end{cases}$$
and $\alpha_{1}=\alpha_{s}=0.$
\end{definition}

\begin{definition}
For $1\leq i\leq s$,
$$\tau_{i}=\tau(S_{i})=
\begin{cases}
1 & \text{ if } S_{i} \text{ is the internal segment containing an edge connecting }\\
  & \text{ the vertices of degree 3 and }l(S_{i})=3,\\
0 & \text{ otherwise. }
\end{cases}$$
\end{definition}

We call the vectors $\alpha=(\alpha_{1},\alpha_{2},...,\alpha_{s})$, $\beta=(\beta_{1},\beta_{2},...,\beta_{s})$, $\tau=(\tau_{1},\tau_{2},...,\tau_{s})$ as \textit{structural vectors}. Note that the structural vectors $\alpha=(\alpha_{1},\alpha_{2},...,\alpha_{s})$ and $\beta=(\beta_{1},\beta_{2},...,\beta_{s})$ can be obtained from the length vector $l=(l_{1},l_{2},...,l_{s})$. Now, we are in position to establish general expression for calculating the BID indices of polyomino chains.
\begin{thm}\label{t1}
Let $B_{n}$ be any polyomino chain having $n\geq3$ squares and $s$ segment(s) $S_{1}, S_{2},S_{3},...,S_{s}$ with the length vector $l=(l_{1},l_{2},...,l_{s})$ and structural vector $\tau=(\tau_{1},\tau_{2},...,\tau_{s})$. Then
\begin{eqnarray*}
TI(B_{n})&=&3n\theta_{3,3}+(2\theta_{2,3}-6\theta_{3,3}+4\theta_{3,4})s
+(2\theta_{2,2}+2\theta_{2,3}+\theta_{3,3}-4\theta_{3,4})\\
&&+ \ (\theta_{2,4}-\theta_{2,3}+\theta_{3,3}-\theta_{3,4})[\beta_{1}+\beta_{s}]+(\theta_{3,3}-2\theta_{3,4}+\theta_{4,4})\sum_{i=1}^{s}\tau_{i}\\
&&+ \ (2\theta_{2,4}-2\theta_{2,3}+3\theta_{3,3}-4\theta_{3,4}+\theta_{4,4})\sum_{i=1}^{s}\alpha_{i}.
\end{eqnarray*}
\end{thm}

\begin{proof}
For $s=1,2$ the result can be easily verified, so we assume that $s\geq3$. For $1\leq i\leq s$, suppose that $E_{1}(S_{i})$ is the set of those edges of the segment $S_{i}$ which are cut across by the straight dashed line passing through the centre of $S_{i}$ and let $E_{2}(S_{i})= \{\text{Bold edges of the segment} \ S_{i}\}=$ The set of all those edges of the segment $S_{i}$ which are not cut across by any straight dashed line (see the Figure \ref{fig:c}), then
$$E(B_{n})=\left(\bigcup_{i=1}^{s}E_{1}(S_{i})\right)\cup\left(\bigcup_{i=1}^{s}E_{2}(S_{i})\right).$$
It can be easily seen that $E_{1}(S_{1}), E_{1}(S_{2}),..., E_{1}(S_{s}), E_{2}(S_{1}), E_{2}(S_{2}),..., E_{s}(S_{s})$ are pairwise disjoint. Since $B_{n}$ contains only vertices of degree 2,3 and 4, hence from Equation (\ref{z}) it follows that

\begin{figure}
   \centering
    \includegraphics[width=3.5in, height=2.8in]{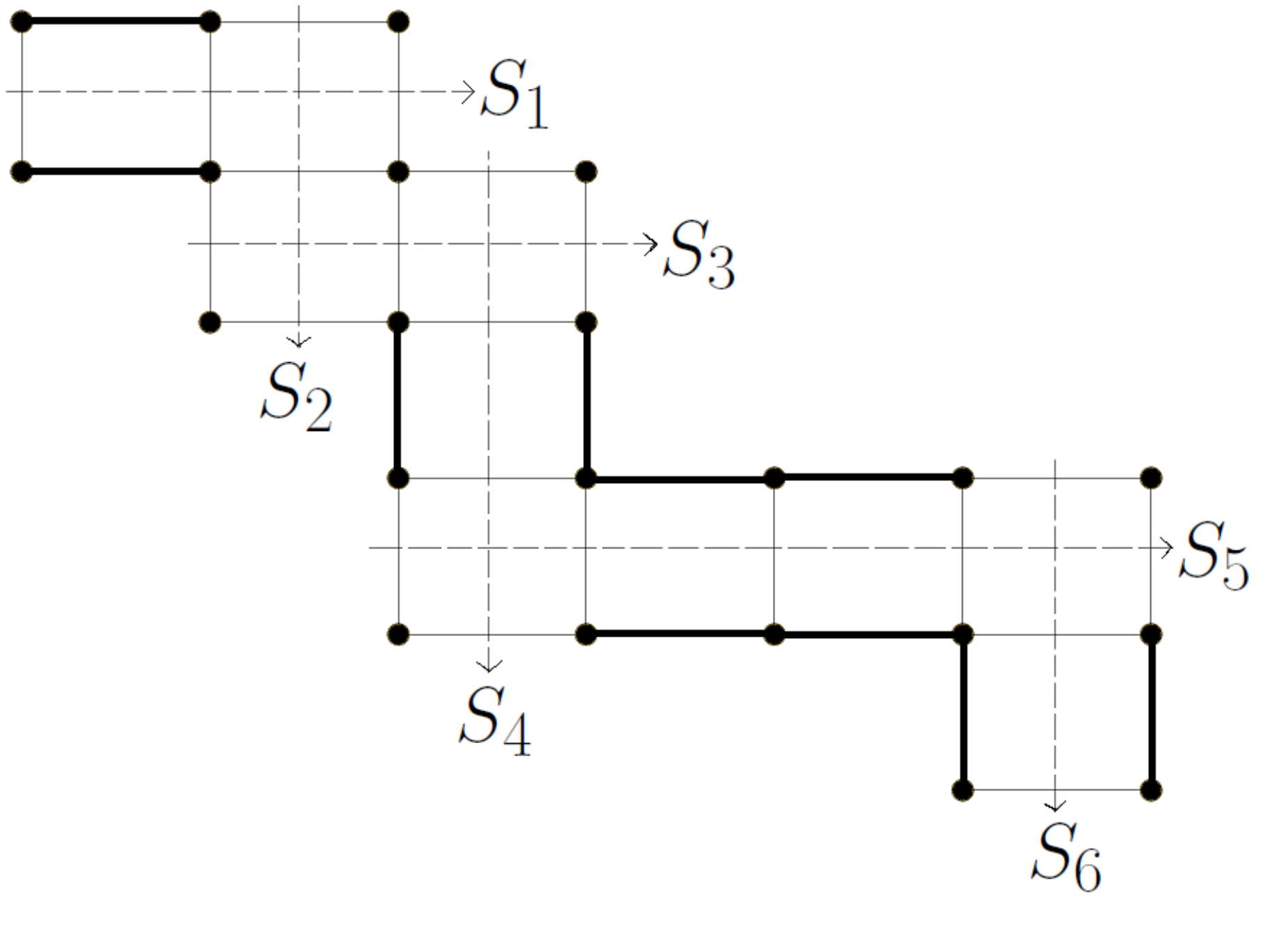}
    \caption{Partition of the edges of a polyomino chain}
    \label{fig:c}
      \end{figure}

\begin{equation}\label{y}
TI(B_{n})=\displaystyle\sum_{2\leq a\leq b\leq 4}x_{a,b}(B_{n}).\theta_{a,b} \ .
\end{equation}
Now, we calculate $x_{a,b}(B_{n})$ for $2\leq a\leq b\leq4$. It is easy to see that $x_{2,2}(B_{n})=2$. For $r=1,2$ and $1\leq i\leq s$, let $x_{a,b}^{(r)}(S_{i})$ is the number of those edges of the segment $S_{i}$ which connect the vertices of degrees $a,b$ and belong to the set $E_{r}(S_{i})$, then
$$x_{2,3}^{(1)}(S_{1})=1-\alpha_{2}, \  x_{2,3}^{(1)}(S_{s})=1-\alpha_{s-1}$$
and for $2\leq i\leq s-1$,
$$x_{2,3}^{(1)}(S_{i})=2-\alpha_{i-1}-\alpha_{i+1}.$$
Furthermore,
$$x_{2,3}^{(2)}(S_{1})=2-\beta_{1}, \ x_{2,3}^{(2)}(S_{s})=2-\beta_{s}\ \ \ \text{and for} \ \ \ 2\leq i\leq s-1, \ x_{2,3}^{(2)}(S_{i})=0.$$
Hence by summing the all $x_{2,3}^{(r)}(S_{i})$ over $r=1,2$ and $1\leq i\leq s$, one have
\begin{eqnarray*}
x_{2,3}(B_{n})&=&\sum_{i=1}^{s}\displaystyle\sum_{r=1}^{2}x_{2,3}^{(r)}(S_{i})\\
&=&2(s+1)-\beta_{1}-\beta_{s}-\sum_{i=1}^{s-1}\alpha_{i}-\sum_{i=2}^{s}\alpha_{i}\\
&=&2(s+1)-\sum_{i=1}^{s}[\alpha_{i}+\beta_{i}].
\end{eqnarray*}
Now, we evaluate $x_{3,4}(B_{n})$ as follows:
$$x_{3,4}^{(1)}(S_{1})=x_{3,4}^{(1)}(S_{s})=1 \ \text{and for} \ 2\leq i\leq s-1, \ x_{3,4}^{(1)}(S_{i})=2-2\beta_{i}.$$
Moreover,
$$x_{3,4}^{(2)}(S_{1})=1-\beta_{1}, \ x_{3,4}^{(2)}(S_{s})=1-\beta_{s} \ \text{and for} \ 2\leq i\leq s-1, \ x_{3,4}^{(2)}(S_{i})=2(1-\beta_{i}-\tau_{i}).$$
But,
\begin{eqnarray*}
x_{3,4}(B_{n})&=&\sum_{i=1}^{s}\displaystyle\sum_{r=1}^{2}x_{3,4}^{(r)}(S_{i})\\
&=&4-\beta_{1}-\beta_{s}+4\sum_{i=2}^{s-1}[1-\beta_{i}]-2\sum_{i=2}^{s-1}\tau_{i}\\
&=&4(s-1)+3\beta_{1}+3\beta_{s}-4\sum_{i=1}^{s}\beta_{i}-2\sum_{i=1}^{s}\tau_{i}.
\end{eqnarray*}
%By using similar technique, one have
In a similar way, one have
$$x_{2,4}(B_{n})=\sum_{i=1}^{s}[\alpha_{i}+\beta_{i}] \ \ \ \text{and} \ \ \ x_{4,4}(B_{n})=\sum_{i=1}^{s}[\alpha_{i}+\tau_{i}].$$
Lastly, the relation $\mid E(B_{n})\mid \ =\sum_{2\leq a\leq b\leq4}x_{a,b}(B_{n})=3n+1$ implies that
$$x_{3,3}(B_{n})=3n-6s+1+\beta_{1}+\beta_{s}+3\sum_{i=1}^{s}\alpha_{i}+\sum_{i=1}^{s}\tau_{i}.$$
After substituting the values of $x_{a,b}(B_{n})$ (where $2\leq a\leq b\leq4$) in Equation (\ref{y}), we arrive at the desired result.
\end{proof}
Since the linear chain $L_{n}$ and zigzag chain $Z_{n}$ has 1 and $n-1$ segment(s) respectively, the following corollary is a direct consequence of Theorem \ref{t1}.
\begin{cor}
Let $L_{n}$ and $Z_{n}$ be linear and zigzag chains respectively with $n\geq3$ squares. Then
$$TI(L_{n})=2\theta_{2,2}+4\theta_{2,3}+(3n-5)\theta_{3,3}$$
$$TI(Z_{n})=2\theta_{2,2}+4\theta_{2,3}+(2n-4)\theta_{2,4}+2\theta_{3,4}+(n-3)\theta_{4,4}.$$
\end{cor}

Let us denote by $\Omega_{n}$ the collection of all those polyomino chains $B_{n}$ in which no internal segment of length three has edge connecting the vertices of degree three. The Equation (\ref{z}) gives the first Zagreb index $M_{1}$ (respectively second Zagreb index $M_{2}$) for $\theta_{a,b}=a+b$ (respectively $\theta_{a,b}=ab$). The following result reported by Yarahmadi \textit{et al.} \cite{z} follows from Theorem \ref{t1}.
\begin{cor}\cite{z}
Let $B_{n}\in\Omega_{n}$ be a polyomino chain having $n\geq3$ squares and $s$ segments $S_{1}, S_{2},S_{3},...,S_{s}$ with length vector $l=(l_{1},l_{2},...,l_{s})$. Then
$$M_{1}(B_{n})=18n+2s-4,$$
$$M_{2}(B_{n})=27n+6s-19-\sum_{i=1}^{s}\beta_{i}.$$
\end{cor}
If we take $\theta_{a,b}=\mid a-b\mid$ in Equation (\ref{z}), then $TI$ is the Albertson index $A$ and hence we have another consequence of Theorem \ref{t1}:
\begin{cor}\cite{y2}
For $n\geq3$, let $B_{n}\in\Omega_{n}$ be a polyomino chain having $n$ squares and $s$ segments $S_{1}, S_{2},S_{3},...,S_{s}$ with length vector $l=(l_{1},l_{2},...,l_{s})$. Then
$$A(B_{n})=6s-2-2\displaystyle\sum_{i=1}^{s}\alpha_{i}.$$
\end{cor}
The choice $\theta_{a,b}=\frac{2}{a+b}$ in Equation (\ref{z}), corresponds to the harmonic index $H$. Recently, Deng \textit{et al.} \cite{deng} obtained the following result which can be deduced from Theorem \ref{t1}.
\begin{cor}\cite{deng}
If $B_{n}\in\Omega_{n}$ is a polyomino chain with $n\geq3$ squares and $s$ segments $S_{1}, S_{2},S_{3},...,S_{s}$ with lengths $l_{1},l_{2},...,l_{s}$ respectively. Then
\[H(B_{n})=
\begin{cases}
n-\frac{2}{35}s-\frac{11}{420}t+\frac{20}{21} & \text{if $l_{1}=l_{s}=2$,} \\
n-\frac{2}{35}s-\frac{11}{420}t+\frac{104}{105} & \text{if $l_{1},l_{s}>2$,} \\
n-\frac{2}{35}s-\frac{11}{420}t+\frac{34}{35} & \text{otherwise,}
\end{cases}\]
where $t$ is the number of segments of length two among $\{S_{2},S_{3},...,S_{s-1}\}$.
\end{cor}
The general Randi\'{c} index $R_{\gamma}$ can be obtained from Equation (\ref{z}) if one take $\theta_{a,b}=(ab)^{\gamma}$ where $\gamma$ is non zero real number. Very recently, An and Xiong \cite{An2} derived an efficient formula (given in Corollary \ref{ccr0}) to calculate the general Randi\'{c} index of polyomino chains. Bearing in mind the fact \[\sum_{i=2}^{s-1}\left[3+\beta_{i}\right]^{\gamma}
=\sum_{i=2}^{s-1}\left[3^{\gamma}+\left(4^{\gamma}-3^{\gamma}\right)\beta_{i}\right],\]
we can obtain the aforementioned formula from Theorem \ref{t1}.
\begin{cor}\label{ccr0}\cite{An2}
Let $B_{n}\in\Omega_{n}$ be a polyomino chain with $n\geq3$ squares and consisting of
$s$ segments $S_{1}, S_{2},S_{3},...,S_{s}$ with lengths $l_{1},l_{2},...,l_{s}$ respectively. Let $\gamma\geq1$ be an arbitrary real number. Then
\[R_{\gamma}(B_{n})=
\begin{cases}
A+\left(4.12^{\gamma}-6.9^{\gamma}\right)s-2.6^{\gamma} & \text{if $s=1$,} \\
A+\left(4.12^{\gamma}-6.9^{\gamma}\right)s+\left(9^{\gamma}+8^{\gamma}-12^{\gamma}-6^{\gamma}\right)\\
\times\left[\beta_{1}+\beta_{s}\right]+\left(16^{\gamma}+3.9^{\gamma}-4.12^{\gamma}\right)\sum_{i=2}^{s-1}\beta_{i}\\
2^{\gamma+1}\sum_{i=2}^{s-1}\left[3+\beta_{i}\right]^{\gamma} & \text{otherwise,}
\end{cases}\]
where $A=(3n+1).9^{\gamma}-4.12^{\gamma}+6^{\gamma+1}+2.4^{\gamma}.$
\end{cor}

To characterize the extremal polyomino chains with respect to BID indices, let us suppose that
$$\Theta_{1}=2\theta_{2,3}-6\theta_{3,3}+4\theta_{3,4}, \ \Theta_{2}=\theta_{2,4}-\theta_{2,3}+\theta_{3,3}-\theta_{3,4},$$ $$\Theta_{3}=2\theta_{2,4}-2\theta_{2,3}+3\theta_{3,3}-4\theta_{3,4}+\theta_{4,4} \text{  and  } \Theta_{4}=\theta_{3,3}-2\theta_{3,4}+\theta_{4,4}.$$
Furthermore, let $\Psi_{TI}(S_{1})=\Theta_{1}+\Theta_{2}\beta_{1}, \ \Psi_{TI}(S_{s})=\Theta_{1}+\Theta_{2}\beta_{s}$ and for $s\geq3$, assume that $\Psi_{TI}(S_{i})=\Theta_{1}+\Theta_{3}\alpha_{i}+\Theta_{4}\tau_{i}$ where $2\leq i\leq s-1$. Then
\begin{equation}\label{Eq.11}
\Psi_{TI}(B_{n})=\sum_{i=1}^{s}\Psi_{TI}(S_{i})=\Theta_{1}s+\Theta_{2}(\beta_{1}+\beta_{s})
+\Theta_{3}\sum_{i=1}^{s}\alpha_{i}+\Theta_{4}\sum_{i=1}^{s}\tau_{i} \ .
\end{equation}
Hence the formula given in Theorem \ref{t1} can be rewritten as
\begin{equation}\label{Eq.12}
TI(B_{n})=3n\theta_{3,3}+(2\theta_{2,2}+2\theta_{2,3}+\theta_{3,3}-4\theta_{3,4})+\Psi_{TI}(B_{n}).
\end{equation}
Therefore, keeping the relation (\ref{Eq.12}) in mind, one have the following straightforward but important lemma for characterizing the extremal polyomino chains.

\begin{lem}\label{L1}
For any polyomino chain $B_{n}$ having $n\geq3$ squares, $TI(B_{n})$ is maximum (respectively minimum) if and only if $\Psi_{TI}(B_{n})$ is maximum (respectively minimum).
\end{lem}

\begin{thm}\label{t2}
Let $B_{n}$ be any polyomino chain with $n\geq3$ squares.\\ \\
\textit{(1).} If $\Theta_{1}>0$ and $\Theta_{1}+2\Theta_{i}>0$ for $i=2,3,4$, then $TI(B_{n})$ is minimum if and only if $B_{n}\cong L_{n}$.\\ \\
\textit{(2).} If $\Theta_{1}<0$ and $\Theta_{1}+2\Theta_{i}<0$ for $i=2,3,4$, then $TI(B_{n})$ is maximum if and only if $B_{n}\cong L_{n}$.
\end{thm}

\begin{proof}
\textit{(1).} Suppose that $B_{n}$ has $s$ segments $S_{1}, S_{2},S_{3},...,S_{s}$ with the length vector $l=(l_{1},l_{2},...,l_{s})$ and structural vector $\tau=(\tau_{1},\tau_{2},...,\tau_{s})$. If $s\geq2$ then
\[\Psi_{TI}(S_{1})+\Psi_{TI}(S_{s})=2\Theta_{1}+\Theta_{2}[\beta(S_{1})+\beta(S_{s})]>\Theta_{1},\]
the last inequality follows from the facts $\beta(S_{1})+\beta(S_{s})\leq2$ and $\Theta_{1}+2\Theta_{2}>0$. Also, the inequalities $\Theta_{1}>0$, $\Theta_{1}+2\Theta_{3}>0$ and $\Theta_{1}+2\Theta_{4}>0$ implies that $\Theta_{1}+\Theta_{3}>0$ and $\Theta_{1}+\Theta_{4}>0$. Hence for $2\leq i\leq s-1$ (if $s\geq3$), the quantity $\Psi_{TI}(S_{i})$ must be positive. Therefore, for $s\geq2$
\[\Psi_{TI}(B_{n})=\sum_{i=1}^{s}\Psi_{TI}(S_{i})>\Theta_{1}=\Psi_{TI}(L_{n}).\]
By using Lemma \ref{L1}, we have $TI(B_{n})\geq TI(L_{n})$ with equality if and only if $B_{n}\cong L_{n}$.\\

\textit{(2).} The proof is fully analogous to that of part \textit{(1)}.

\end{proof}

Equation (\ref{z}) gives the first geometric-arithmetic index for $\theta_{a,b}= \frac{2\sqrt{ab}}{a+b}$,
Randi$\acute{c}$ index for $\theta_{a,b}= \frac{1}{\sqrt{ab}}$ and sum-connectivity index for $\theta_{a,b}= \frac{1}{\sqrt{a+b}}$. Rada \cite{r4} recently proved that the linear chain $L_{n}$ has the extremal value for many well known topological indices (including the aforementioned indices). This result can also be deduced from Theorem \ref{t2}:
\begin{cor}\cite{r4}
Among all polyomino chains with $n$ squares, the linear chain $L_{n}$ has the maximum Randi\'{c} index,
maximum sum-connectivity index, maximum harmonic index, maximum geometric-arithmetic
index, minimum first Zagreb index and minimum second Zagreb index.
\end{cor}

\begin{proof}
Routine computations yield that all $\Theta_{1},\Theta_{2},\Theta_{3},\Theta_{4}$ satisfy the hypothesis of Theorem \ref{t2}(2) for the Randi\'{c} index, sum-connectivity index, harmonic index and geometric-arithmetic index. Moreover, $\Theta_{1},\Theta_{2},\Theta_{3},\Theta_{4}$ satisfy the hypothesis of Theorem \ref{t2}(1) for the first Zagreb index and second Zagreb index. Therefore, by virtue of Theorem \ref{t2}, one have the desired result.
\end{proof}

\begin{thm}\label{t22}
Let $B_{n}\in\Omega_{n}$ be a polyomino with $n\geq3$ squares.\\ \\
\textit{(1).} If $\Theta_{1}$, $\Theta_{1}+2\Theta_{2}$ and $\Theta_{1}+2\Theta_{3}$ are all positive, then
\[TI(L_{n})\leq TI(B_{n})\leq TI(Z_{n}).\]
Right (respectively left) equality holds if and only if $B_{n}\cong Z_{n}$ (respectively $B_{n}\cong L_{n}$).\\ \\
\textit{(2).} If $\Theta_{1}$, $\Theta_{1}+2\Theta_{2}$ and $\Theta_{1}+2\Theta_{3}$ are all negative, then
\[TI(Z_{n})\leq TI(B_{n})\leq TI(L_{n}).\]
Right (respectively left) equality holds if and only if $B_{n}\cong L_{n}$ (respectively $B_{n}\cong Z_{n}$).

\end{thm}

\begin{proof}
\textit{(1).} The proof of lower bound is analogous to that of Theorem \ref{t2}\textit{(1)}. To prove the upper bound let us suppose that for the polyomino chain $B_{n}^{*}\in\Omega_{n}$, $\Psi_{TI}(B_{n}^{*})$ is maximum. Let $B_{n}^{*}$ has $s$ segments $S_{1}, S_{2},...,S_{s}$ with the length vector $(l_{1},l_{2},...,l_{s})$. Simple calculations yield that $\Psi_{TI}(Z_{n})> \Psi_{TI}(L_{n})$, which means that $s$ must be greater than 1.

If at least one of external segments of $B_{n}^{*}$ has length greater than 2. Without loss of generality, assume that $l_{1}\geq3$. Then it can be easily seen that there exist a polyomino chain $B_{n}^{(1)}\in\Omega_{n}$ having length vector $(2,l_{1}-1,l_{2},...,l_{s})$ and
\[\Psi_{TI}(B_{n}^{(1)})-\Psi_{TI}(B_{n}^{*})=\left(\frac{\Theta_{1}}{2}+\Theta_{2}\right)+\left(\frac{\Theta_{1}}{2}+x\Theta_{3}\right)>0, \ \ \ (\text{where $x=0$ or 1})\]
which is a contradiction to the definition of $B_{n}^{*}$. Hence both external segments of $B_{n}^{*}$ must have length 2.

If some internal segment of $B_{n}^{*}$ has length greater than 2, say $l_{j}\geq3$ where $2\leq j\leq s-1$ and $s\geq3$. Then there exist a polyomino chain $B_{n}^{(2)}\in\Omega_{n}$ having length vector $(l_{1},l_{2},...,l_{j-1},2,l_{j}-1,...,l_{s})$ and
\[\Psi_{TI}(B_{n}^{(2)})-\Psi_{TI}(B_{n}^{*})=\left(\frac{\Theta_{1}}{2}+\Theta_{3}\right)+\left(\frac{\Theta_{1}}{2}+y\Theta_{3}\right)>0, \ \ \ (\text{where $y=0$ or 1})\]
again a contradiction. Hence every internal segment of $B_{n}^{*}$ has length 2. Therefore, $B_{n}^{*}\cong Z_{n}$ and by Lemma \ref{L1} desired result follows. \\ \\
\textit{(2).} The proof is fully analogous to that of Part \textit{(1).}

\end{proof}

Recall that $\Theta_{1}=2,\Theta_{2}=\Theta_{3}=0$ for the first Zagreb index $M_{1}$ and $\Theta_{1}=6,\Theta_{2}=\Theta_{3}=-1$ for the second Zagreb index $M_{2}$. Hence the following corollary follows from Theorem \ref{t22}\textit{(1)}:
\begin{cor}\cite{z}
Let $B_{n}\in\Omega_{n}$ be a polyomino chain having $n\geq3$ squares, then
$$M_{i}(L_{n}) \leq M_{i}(B_{n})\leq M_{i}(Z_{n}), \ \ \text{ (where } i=1,2)$$
with left (respectively right) equality if and only if $B_{n}\cong L_{n}$ (respectively $B_{n}\cong Z_{n}$).
\end{cor}

Since all $\Theta_{1},\Theta_{2},\Theta_{3}$ are negative for the harmonic index $H$, hence from Theorem \ref{t22}\textit{(2)} we have:
\begin{cor}\cite{deng}
Let $B_{n}\in\Omega_{n}$ be any polyomino chain having $n\geq3$ squares, then
$$H(Z_{n}) \leq H(B_{n})\leq H(L_{n})$$
with left (respectively right) equality if and only if $B_{n}\cong Z_{n}$ (respectively $B_{n}\cong L_{n}$).
\end{cor}

Equation (\ref{z}) gives the natural logarithm of the multiplicative sum Zagreb index for $\theta_{a,b}=ln(a+b)$ and natural logarithm of the multiplicative second Zagreb index for $\theta_{a,b}=lna+lnb$ where $ln$ denotes the natural logarithm. The following corollary is an immediate consequence of Theorem \ref{t22}:

\begin{cor}\label{c1}
Let $B_{n}\in\Omega_{n}$ be a polyomino chain having $n\geq3$ squares.\\ \\
\textit{(1).} If $TI$ is one of the following indices: first geometric-arithmetic index,
Randi$\acute{c}$ index, sum-connectivity index. Then
$$TI(Z_{n}) \leq TI(B_{n})\leq TI(L_{n}),$$
with right (respectively left) equality if and only if $B_{n}\cong L_{n}$ (respectively $B_{n}\cong Z_{n}$).\\ \\
\textit{(2).} For the multiplicative sum Zagreb index $\Pi_{1}^{*}$ and multiplicative second Zagreb index $\Pi_{2}$, the following inequalities hold
$$\Pi_{1}^{*}(L_{n}) \leq \Pi_{1}^{*}(B_{n})\leq \Pi_{1}^{*}(Z_{n})$$
$$\Pi_{2}(L_{n}) \leq \Pi_{2}(B_{n})\leq \Pi_{2}(Z_{n})$$
with left (respectively right) equalities if and only if $B_{n}\cong L_{n}$ (respectively $B_{n}\cong Z_{n}$).
\end{cor}

\begin{proof}
\textit{(1).} Simple calculations show that $\Theta_{1},\Theta_{2}$ and $\Theta_{3}$ are all negative for the first geometric-arithmetic index, Randi$\acute{c}$ index and sum-connectivity index. Hence from Theorem \ref{t22}\textit{(2)} desired result follows.\\ \\
\textit{(2).} It is easy to see that $\Theta_{1},\Theta_{2}$ and $\Theta_{3}$ are all positive for the natural logarithm of $\Pi_{1}^{*}$. Also, $\Theta_{1}=0.3398,\Theta_{2}=0,\Theta_{3}=-0.0001$ for the natural logarithm of $\Pi_{2}$. Therefore, by virtu of Theorem \ref{t22}\textit{(1)}, one have
$$ln[\Pi_{1}^{*}(L_{n})] \leq ln[\Pi_{1}^{*}(B_{n})]\leq ln[\Pi_{1}^{*}(Z_{n})]$$
$$ln[\Pi_{2}(L_{n})] \leq ln[\Pi_{2}(B_{n})]\leq ln[\Pi_{2}(Z_{n})]$$
with left (respectively right) equality if and only if $B_{n}\cong L_{n}$ (respectively $B_{n}\cong Z_{n}$). Since the exponential function is inverse of the natural logarithm function and is strictly increasing, hence the required result follows from above inequalities.
\end{proof}

If we replace $\theta_{a,b}$ with $\left(\frac{a+b-2}{ab}\right)^{\gamma}$ and $\left(a+b\right)^{\gamma}$ (where $\gamma$ is a non-zero real number) in Equation (\ref{z}), then $TI$ corresponds to the general atom-bond connectivity index $ABC_{\gamma}$ and general sum-connectivity index $\chi_{\gamma}$ respectively. The following result is another consequence of Theorem \ref{t22}:

\begin{cor}\label{c111}
If $B_{n}\in\Omega_{n}$ is any polyomino chain with $n\geq3$ squares, then
$$R_{\gamma}(L_{n}) \leq R_{\gamma}(B_{n})\leq R_{\gamma}(Z_{n}), \ \gamma>0$$
$$ABC_{\gamma}(L_{n}) \leq ABC_{\gamma}(B_{n})\leq ABC_{\gamma}(Z_{n}), \ \gamma>1$$
$$\chi_{\gamma}(L_{n}) \leq \chi_{\gamma}(B_{n})\leq \chi_{\gamma}(Z_{n}), \ \gamma>0$$
with left (respectively right) equalities if and only if $B_{n}\cong L_{n}$ (respectively $B_{n}\cong Z_{n}$).
\end{cor}
\begin{proof}
Firstly, let us prove the result for the general Randi\'{c} index $R_{\gamma}$.
In light of Lagrange's mean-value theorem, there exist numbers $c_{1},c_{2}$ such that $2<c_{1}<3<c_{2}<4$ and
\[\Theta_{1}=2\gamma3^{\gamma} c_{2}^{\gamma-1}\left[2-\left(\frac{c_{1}}{c_{2}}\right)^{\gamma-1}\right].\]
It can be easily seen that
\[\left(\frac{c_{1}}{c_{2}}\right)^{\gamma
-1}
\begin{cases}
\leq1  & \text{ if $\gamma\geq1$ ,}\\
<2^{1-\gamma}<2 & \text{ if  $0<\gamma<1$.}
\end{cases}\]
Hence, it follows that if $\gamma$ is positive then $\Theta_{1}>0$. Moreover, there exist numbers $c_{3},c_{4},c_{5}$ such that $3<c_{3}<4<c_{4}<6<c_{5}<8$ and
\[\frac{\Theta_{1}}{2}+\Theta_{3}
=\gamma2^{\gamma}\left(2c_{5}^{\gamma-1}-2c_{4}^{\gamma-1}+c_{3}^{\gamma-1}\right),\]
which is obviously positive for all $\gamma\geq1$. Note that the expression $2c_{5}^{\gamma-1}-2c_{4}^{\gamma-1}+c_{3}^{\gamma-1}$ can be rewritten as
$c_{5}^{\gamma-1}\left[2-\left(\frac{c_{4}}{c_{5}}\right)^{\gamma-1}\right]
+(c_{3}^{\gamma-1}-c_{4}^{\gamma-1})$ where $\left(\frac{c_{4}}{c_{5}}\right)^{\gamma-1}<2^{1-\gamma}<2$ and
$c_{3}^{\gamma-1}>c_{4}^{\gamma-1}$ for $0<\gamma<1$. Hence for $0<\gamma<1$, the quantity $\frac{\Theta_{1}}{2}+\Theta_{3}$ is again positive. Furthermore,
\[\frac{\Theta_{1}}{2}+\Theta_{2}
=\gamma c_{7}^{\gamma-1}\left[3-\left(\frac{c_{6}}{c_{7}}\right)^{\gamma-1}\right],\]
where $8<c_{6}<9<c_{7}<12$. Note that for $0<\gamma<1$, $\left(\frac{c_{6}}{c_{7}}\right)^{\gamma-1}<\left(\frac{12}{8}\right)^{1-\gamma}<3$ and for $\gamma\geq1$, $\left(\frac{c_{6}}{c_{7}}\right)^{\gamma-1}\leq1$. Hence $\frac{\Theta_{1}}{2}+\Theta_{2}$ is positive for $\gamma>0$. Therefore, from Theorem \ref{t22}, desired result follows.

The proofs of the remaining inequalities are fully analogous and hence we omit.

\end{proof}

\begin{rk}
Recently, An and Xiong \cite{An2} characterized the extremal polyomino chains for the general Randi\'{c} index $R_{\gamma}$ for $\gamma\geq1$. The first inequality of Corollary \ref{c111} can be considered as a generalized version of one given in \cite{An2}.
\end{rk}

The substitution $\theta_{a,b}=\left(\frac{ab}{a+b-2}\right)^{3}$ in Equation (\ref{z}), gives augmented Zagreb index $AZI$.

\begin{thm}\label{t3}
If $B_{n}$ is any polyomino chain with $n\geq3$ squares, then
$$AZI(L_{n}) \leq AZI(B_{n})$$
with equality if and only if $B_{n}\cong L_{n}$.
\end{thm}

\begin{proof}
Suppose that $B_{n}$ has $s$ segments $S_{1}, S_{2},S_{3},...,S_{s}$ with the length vector $l=(l_{1},l_{2},...,l_{s})$ and structural vector $\tau=(\tau_{1},\tau_{2},...,\tau_{s})$. Straightforward computations yield \[\Theta_{1}\approx2.9523,\Theta_{2}\approx-2.4334,\Theta_{3}\approx-2.1612,\Theta_{4}\approx2.7056.\]
Firstly, we prove the lower bound. It can be easily verified that $\Psi_{TI}(Z_{n})>\Psi_{AZI}(L_{n})$ and hence we take $B_{n}\not\cong Z_{n}$. Let $s\geq2$ then by definition of $\Psi_{AZI}$, the quantities $\Psi_{AZI}(S_{1})+\Psi_{AZI}(S_{s})$ and $\Psi_{AZI}(S_{i})$ (where $2\leq i\leq s-1$ if $s\geq3$) must be positive. If at least one external segment has length greater than 2, then
\[\Psi_{AZI}(S_{1})+\Psi_{AZI}(S_{s})=2\Theta_{1}+\Theta_{2}(\beta_{1}+\beta_{s})>\Theta_{1}.\]
If some internal segment has length greater than 2, say $l_{i}\geq3$ where $2\leq i\leq s-1$ and $s\geq3$, then $\Psi_{AZI}(S_{i})\geq\Theta_{1}.$
In both cases,
\[\Psi_{AZI}(B_{n})=\sum_{i=1}^{s}\Psi_{AZI}(S_{i})>\Theta_{1}=\Psi_{AZI}(L_{n}).\]
By virtue of Lemma \ref{L1}, $AZI(L_{n})\leq AZI(B_{n})$ with equality if and only if $B_{n}\cong L_{n}$.
\end{proof}

At this time, the problem of finding polyomino chain with maximum AZI value over
the class of all polyomino chains with fixed number of squares seems to be difficult and we leave it for future
work. However, here we prove some structural properties of the polyomino chain
having maximum AZI value.

\begin{thm}\label{t3}
If $B_{n}^{+}$ is the polyomino chain with $n\geq6$ squares and maximum AZI value. Then the following properties hold.\\ \\
\textit{(1).} Every segment of $B_{n}^{+}$ has length less than 4 (and consequently $B_{n}^{+}$ has at least 3 segments).\\ \\
\textit{(2).} No two segments of $B_{n}^{+}$ with lengths 2 are consecutive.\\ \\
\textit{(3).} If at least one external segment of $B_{n}^{+}$ has length 2, then no two internal segments with lengths 3 are consecutive.\\ \\
\textit{(4).} If an external segment of $B_{n}^{+}$ has length 3, then its adjacent segment has also length 3.
\end{thm}
\begin{proof}
Bearing in mind the Lemma \ref{L1}, one can say that $\Psi_{AZI}(B_{n}^{+})$ is maximum. Let $B_{n}^{+}$ has $t$ segments $S_{1}^{+}, S_{2}^{+},...,S_{t}^{+}$ with length vector $(l_{1}^{+},l_{2}^{+},...,l_{t}^{+})$ and structural vector $(\tau_{1}^{+},\tau_{2}^{+},...,\tau_{t}^{+})$ where $l_{i}^{+}=l(S_{i}^{+})$ and $\tau_{i}^{+}=\tau(S_{i}^{+})$ for $1\leq i\leq t$. Recall that \[\Theta_{1}\approx2.9523,\Theta_{2}\approx-2.4334,\Theta_{3}\approx-2.1612,\Theta_{4}\approx2.7056.\]
\textit{Proof of Part 1.} Suppose to the contrary that for some $j$ (where $1\leq j\leq t$), the length $l_{j}^{+}$ is greater than 3. If $2\leq j\leq t$, then let us assume that $B_{n}^{(1)}$ be the polyomino chain with length vector $(l_{1}^{+},l_{2}^{+},...,l_{j-1}^{+},3,l_{j}^{+}-2,l_{j+1}^{+}...,l_{t}^{+})$ and structural vector $(\tau_{1}^{+},\tau_{2}^{+},...,\tau_{j-1}^{+},0,0,\tau_{j+1}^{+},...,\tau_{t}^{+})$. Then
\[\Psi_{AZI}(B_{n}^{(1)})-\Psi_{AZI}(B_{n}^{+})=
\begin{cases}
\Theta_{1}+x_{1}\Theta_{2}>0  & \text{if $j=t$},\\
\Theta_{1}+x_{1}\Theta_{3}>0  & \text{otherwise},
\end{cases} \]
where $x_{1}=0$ or 1. This is a contradiction to the maximality of of $\Psi_{AZI}(B_{n}^{+})$. If $j=1$, then for the polyomino chain
$B_{n}^{(2)}$ having length vector $(3,l_{1}^{+}-2,l_{2}^{+},l_{3}^{+},...,l_{t}^{+})$ and structural vector $(0,0,\tau_{2}^{+},\tau_{3}^{+},...,\tau_{t}^{+})$, one have $\Psi_{AZI}(B_{n}^{(2)})-\Psi_{AZI}(B_{n}^{+})>0$, again a contradiction.\\
\textit{Proof of Part 2.} Contrarily assume that $l_{j}^{+}=l_{j+1}^{+}=2$ for some $j$ (where $1\leq j\leq t-1$ and $t\geq3$). Let $B_{n}^{(3)}$ be the polyomino chain obtained from $B_{n}^{+}$ by replacing the segments $S_{j}^{+},S_{j+1}^{+}$ with one having length 3. Then
\[\Psi_{AZI}(B_{n}^{(3)})-\Psi_{AZI}(B_{n}^{+})=
\begin{cases}
x_{2}\Theta_{4}-\Theta_{1}-2\Theta_{3} & \text{if both $S_{j}^{+},S_{j+1}^{+}$ are internal},\\
-(\Theta_{1}+\Theta_{2}+\Theta_{3}) & \text{otherwise},
\end{cases}\]
where $x_{2}=0$ or 1. In all the cases the quantity $\Psi_{AZI}(B_{n}^{(3)})-\Psi_{AZI}(B_{n}^{+})$ is positive and hence a contradiction is obtained.\\
\textit{Proof of Part 3.} Let us suppose, to the contrary, that $l_{j}^{+}=l_{j+1}^{+}=3$ for some $j$ (where $2\leq j\leq t-2$ and $t\geq4$). Since at least one of $\tau_{j}^{+},\tau_{j+1}^{+}$ is 0, without loss of generality assume that $\tau_{j}^{+}=0$. Since at least one of $l_{1}^{+},l_{t}^{+}$ is 2. Without loss of generality, we suppose that $l_{1}^{+}=2$. Let $B_{n}^{(4)}$ be the polyomino chain obtained from $B_{n}^{+}$ by interchanging the segments $S_{1}^{+}$ and $S_{j}^{+}$. Then
\[\Psi_{AZI}(B_{n}^{(4)})-\Psi_{AZI}(B_{n}^{+})=\Theta_{3}-\Theta_{2}>0,\]
which is a contradiction to the maximality of of $\Psi_{AZI}(B_{n}^{+})$.\\
\textit{Proof of Part 4.} We consider two cases:

\textit{Case 1.} If $l_{1}^{+}=3$. Contrarily suppose that $l_{2}^{+}\neq3$. Then by virtue of Part 1, $l_{2}^{+}=2$ and $l_{t}^{+}\leq3$. Here we have two subcases:

\textit{Subcase 1.1.} If $l_{t}^{+}=2$. Then for the polyomino chain
$B_{n}^{(6)}$ having length vector $(l_{2}^{+}+1,l_{3}^{+},l_{4}^{+},...,l_{t-1}^{+},l_{t}^{+}+1)$ and structural vector $(0,\tau_{3}^{+},\tau_{4}^{+},...,\tau_{t-1}^{+},0)$, one have
\[\Psi_{AZI}(B_{n}^{(6)})-\Psi_{AZI}(B_{n}^{+})=-(\Theta_{1}+\Theta_{2}+\Theta_{3})>0, \text{ a contradiction.}\]

\textit{Subcase 1.2.} If $l_{t}^{+}=3$. Let $B_{n}^{(7)}$ be the polyomino chain
with the length vector $(2,l_{1}^{+},l_{2}^{+},l_{3}^{+},...,l_{t-1}^{+},l_{t}^{+}-1)$ and structural vector $(0,1,0,\tau_{3}^{+},\tau_{4}^{+},...,\tau_{t-1}^{+},0)$. Then
\[\Psi_{AZI}(B_{n}^{(7)})-\Psi_{AZI}(B_{n}^{+})=\Theta_{1}+2\Theta_{2}+\Theta_{4}>0, \text{ again a contradiction.}\]

\textit{Case 2.} If $l_{t}^{+}=3$. Then we have to show that $l_{t-1}^{+}=3$. Using the same technique as adopted in the Case 1, one can easily prove the desired conclusion.\\

\end{proof}

If we replace $\theta_{a,b}$ with $\sqrt{\frac{a+b-2}{ab}}$ in Equation (\ref{z}), then $TI$ corresponds to the atom-bond connectivity index $ABC$.

\begin{thm}\label{t4}
If $B_{n}$ is any polyomino chain with $n\geq3$ squares, then
\[ ABC(B_{n})\leq ABC(Z_{n}),\]
with equality if and only if $B_{n}\cong Z_{n}$.
\end{thm}

\begin{proof}
Suppose that for the polyomino chain $B_{n}^{'}$, $\Psi_{ABC}(B_{n}^{'})$ is maximum. Let $B_{n}^{'}$ has $s$ segments $S_{1}, S_{2},...,S_{s}$ with length vector $(l_{1},l_{2},...,l_{s})$ and structural vector $(\tau_{1},\tau_{2},...,\tau_{s})$. Simple calculations show that
\[\Theta_{1}\approx-0.0038,\Theta_{2}\approx0.0211,\Theta_{3}\approx0.0303,\Theta_{4}\approx-0.012,\]
and hence $\Psi_{ABC}(Z_{n})>\Psi_{ABC}(L_{n})$, which means that $s$ must be greater than 1. If for some internal segment $S_{j}$, $\tau_{j}\neq0$ where $2\leq j\leq s-1$ and $s\geq3$. Then for the polyomino chain $B_{n}^{(1)}$ having length vector $(l_{1},l_{2},...,l_{s})$ and structural vector $(\tau_{1},\tau_{2},...,,\tau_{j-1},0,\tau_{j+1},...,\tau_{s})$, one have
$\Psi_{ABC}(B_{n}^{(1)})>\Psi_{ABC}(B_{n}^{'})$, a contradiction to the definition of $B_{n}^{'}$. Hence $\tau_{1}=\tau_{2}=...=\tau_{s}=0$. If $l_{j}\geq3$ for some $j$ where $1\leq j\leq s$. Then for the polyomino chain $B_{n}^{(2)}$ having length vector
\[
\begin{cases}
(2,l_{1}-1,l_{2},l_{3},...,l_{s}),  & \text{if $j=1$},\\
(l_{1},l_{2},...,l_{j-1},2,l_{j}-1,l_{j+1},...,l_{s})  & \text{if $s\geq3$ and $2\leq j\leq s-1$},\\
(l_{1},l_{2},...,l_{s-2},l_{s-1},2,l_{s}-1)  & \text{if $j=s$},
\end{cases}\]
and structural vector $\tau=(0,0,...,0)$,
one have $\Psi_{ABC}(B_{n}^{(2)})>\Psi_{ABC}(B_{n}^{'})$ which is again a contradiction. Therefore, $B_{n}^{'}\cong Z_{n}$ and hence by Lemma \ref{L1}, the desired result follows.

\end{proof}

\section{Concluding Remarks}

We have established a general formula, given in Theorem \ref{t1}, for evaluating any BID index of polyomino chains. Then using this formula, we have derived some extremal results for BID indices of polyomino chains and hence obtained some corollaries for several well known BID indices: first and second Zagreb indices, first geometric-arithmetic index, Randi$\acute{c}$ index, sum-connectivity index, harmonic index, multiplicative sum Zagreb index, second multiplicative Zagreb index, general Randi$\acute{c}$ index $R_{\gamma}$ (for $\gamma>0$), general atom-bond connectivity index $ABC_{\gamma}$ (for $\gamma>1$) and general sum-connectivity index $\chi_{\gamma}$ (for $\gamma>0$). Thereby, we have generalized all the results of \cite{An2,deng,z} and some of \cite{y2,r4}. Moreover, we have showed that the linear chain $L_{n}$ and zigzag chain $Z_{n}$ has the minimum augmented Zagreb index and maximum atom-bond connectivity index, respectively, over the collection of all polyomino chains with $n\geq3$ squares. However, till now, there are many open problems related to extremal polyomino chains with respect to BID indices. But the problems of finding polyomino chains having minimum atom-bond connectivity index and maximum augmented Zagreb index over the class of all polyomino chains with fixed number of squares, seems to be interesting.


\begin{thebibliography}{00}

\bibitem{a1} M. O. Albertson, The Irregularity of a Graph, \textit{Ars Combin.} \textbf{46}, (1997) 219-225.

\bibitem{AA2} A. Ali, A. A. Bhatti, Z. Raza, A note on the zeroth-order general Randi\'{c} index of cacti and polyomino chains, \textit{Iran. J. Math. Chem.} \textbf{5}, (2014) 143-152.

\bibitem{y2} A. Ali, A. A. Bhatti, Z. Raza, Some vertex-degree-based topological indices of polyomino chains, \textit{J. Comput. Theor. Nanosci.} \textbf{12}(9), (2015) 2101-2107.

\bibitem{An2} M. An, L. Xiong, Extremal polyomino chains with respect to general Randi\'{c} index, \textit{J. Comb. Optim.}, (2014) DOI 10.1007/s10878-014-9781-6.

\bibitem{b1} B. Bollob\'{a}s, P. Erd\"{o}s, Graphs of extremal weights, \textit{Ars Comb.} \textbf{50}, (1998) 225-233.

\bibitem{d4} H. Deng, J. Yang, F. Xia, A general modeling of some vertex-degree based topological indices in benzenoid systems and phenylenes, \textit{Comput. Math. Appl.} \textbf{61}, (2011) 3017-3023.

\bibitem{deng} H. Deng, S. Balachandran, S. K. Ayyaswamy, Y. B. Venkatakrishnan, The harmonic indices of polyomino chains, \textit{Natl. Acad. Sci. Lett.} \textbf{37}(5), (2014) 451-455.

\bibitem{d2} M.V. Diudea (Ed.),\textit{QSPR/QSAR Studies by Molecular Descriptors}, Nova, Huntington, (2001).

\bibitem{e1} M. Eliasi, I. Gutman, A. Iranmanesh, Multiplicative versions of first Zagreb index, \textit{MATCH Commun. Math. Comput. Chem.} \textbf{68}, (2012) 217-230.

\bibitem{e3} E. Estrada, L. Torres, L. Rodr$\acute{i}$guez, I. Gutman, An atom-bond connectivity index: modelling the enthalpy of formation of alkanes, \textit{Indian J. Chem. A} \textbf{37}, (1998) 849-855.

\bibitem{f2} S. Fajtlowicz, On conjectures of Graffiti-II, \textit{Congr. Numer.} \textbf{60}, (1987) 187-197.

\bibitem{f4} B. Furtula, A. Graovac, D. Vuki\v{c}evi\'{c}, Augmented Zagreb index, \textit{J. Math. Chem.} \textbf{48}, (2010) 370-380.

\bibitem{f5} B. Furtula, I. Gutman, M. Dehmer, On structure-sensitivity of degree-based topological indices, \textit{Appl. Math. Comput.} \textbf{219}, (2013) 8973-8978.

\bibitem{g11} S. W. Golomb, Checker boards and polyominoes, \textit{Amer. Math. Monthly} \textbf{61}, (1954) 675-682.

\bibitem{g22} I. Gutman, B. Ru\v{s}\v{c}i\'{c}, N. Trinajsti\'{c}, C. F. Wilcox, Graph theory and molecular orbitals. XII. Acyclic polyenes, \textit{J. Chem. Phys.} \textbf{62}, (1975) 3399-3405.

\bibitem{g2} I. Gutman, J. To\v{s}ovi\'{c}, Testing the quality of molecular structure descriptors: Vertex-degree-based topological indices, \textit{J. Serb. Chem. Soc.} \textbf{78}, (2013) 805-810.

\bibitem{g3} I. Gutman, N. Trinajsti\'{c}, Graph theory and molecular orbitals. Total $\varphi$-electron energy of alternant hydrocarbons, \textit{Chem. Phys. Lett.} \textbf{17}(4), (1972) 535-538.

\bibitem{g4} I. Gutman, Multiplicative Zagreb indices of trees, \textit{Bull. Int. Math. Virt. Inst.} \textbf{1}, (2011) 13-19.

\bibitem{f1} I. Gutman, B. Furtula (Eds.), \textit{Novel Molecular Structure Descriptors—Theory and Applications vols. I-II}, Univ. Kragujevac, Kragujevac, (2010).

\bibitem{g1} I. Gutman, Degree-based topological indices, \textit{Croat. Chem. Acta} \textbf{86}(4), (2013) 351-361.

\bibitem{k1} D.A. Klarner, Polyominoes, in: J.E. Goodman, J. O' Rourke (Eds.), \textit{Handbook of Discrete and Computational Geometry}, CRC Press LLC, (1997).

\bibitem{l1} A. Mehler, P. Wei, A. L\"{u}cking, A network model of interpersonal alignment, \textit{Entropy} \textbf{12}, (2010) 1440-1483.

\bibitem{d3} L. A. J. M\"{u}ller, K. G. Kugler, A. Graber, M. Dehmer, A network-based approach to classify the three domains of life, \textit{Biol. Direct} \textbf{6}, (2011) 140-141.

\bibitem{r1} J. Rada, R. Cruz, I. Gutman, Vertex-degree-based topological indices of catacondensed hexagonal systems, \textit{Chem. Phys. Lett.} \textbf{ 572}, (2013) 154-157.

\bibitem{r4} J. Rada, The linear chain as an extremal value of VDB topological indices of polyomino chains, \textit{Appl. Math. Sci.} \textbf{8}, (2014) 5133-5143.

\bibitem{r3} M. Randi\'{c}, On characterization of molecular branching, \textit{J. Am. Chem. Soc.} \textbf{97}, (1975) 6609-6615.

\bibitem{c1} R. Todeschini , V. Consonni,  \textit{Molecular Descriptors for Chemoinformatics}, Wiley-VCH, Weinheim, (2009).

\bibitem{v1} D. Vuki\v{c}evi\'{c}, B. Furtula, Topological index based on the ratios of geometrical and arithmetical means of end-vertex degrees of edges, \textit{J. Math. Chem.} \textbf{46}, (2009) 1369-1376.

\bibitem{v2} D. Vuki\v{c}evi\'{c}, M. Ga\v{s}perov, Bond additive modeling 1. Adriatic indices, \textit{Croat. Chem. Acta} \textbf{83}, (2010) 243-260.

\bibitem{Vu10} D. Vuki\v{c}evi\'{c}, Bond additive modeling 2. Mathematical properties of max-min rodeg index, \textit{Croat. Chem. Acta} \textbf{83} (3), (2010) 261-273.

\bibitem{VuDu11} D. Vuki\v{c}evi\'{c}, J. Durdevi\'{c}, Bond additive modeling 10. Upper and lower bounds of bond incident degree indices of catacondensed fluoranthenes,  \textit{Chem. Phys. Lett.} \textbf{515}, (2011) 186-189.

\bibitem{w1} H. Van de Waterbeemd, R. E. Carter, G. Grassy, H. Kubiny, Y. C. Martin, M. S. Tutte, and P. Willet, Glossary of terms used in computational drug design, \textit{Pure Appl. Chem.} \textbf{69}, (1997) 1137-1152.

\bibitem{x1} R. Xing, B. Zhou, Extremal trees with fixed degree sequence for atom-bond connectivity index, \textit{Filomat} 26(4) (2012), 683-688.

\bibitem{z} Z. Yarahmadi, A. R. Ashrafi, S. Moradi, Extremal polyomino chains with respect to Zagreb indices, \textit{Appl. Math. Lett.} \textbf{25}, (2012) 166-171.

\bibitem{z2} B. Zhou, N. Trinajsti\'{c}, On a novel connectivity index, \textit{J. Math. Chem.} \textbf{46}, (2009) 1252-1270.

\bibitem{z3} B. Zhou, N. Trinajsti\'{c}, On general sum-connectivity index, \textit{J. Math. Chem.} \textbf{47}, (2010) 210-218.






\end{thebibliography}
\end{document}